\documentclass{amsart}

\usepackage{amssymb, amsmath, amscd, amsthm}
\usepackage{graphicx}
\usepackage{url}
\usepackage{psfrag}
\usepackage{color}
\usepackage[all]{xy}

\newtheorem{thm}{Theorem}[section]
\newtheorem{lem}[thm]{Lemma}
\newtheorem{cor}[thm]{Corollary}
\newtheorem{prop}[thm]{Proposition}

\theoremstyle{definition}
\newtheorem{defi}[thm]{Definition}
\newtheorem{example}[thm]{Example}

\theoremstyle{remark}
\newtheorem{rem}[thm]{Remark}

\newcommand{\R}{\mathbb R}
\newcommand{\N}{\mathbb N}

\newcommand{\newatop}[2]{\genfrac{}{}{0pt}{1}{#1}{#2}}
\newcommand{\vu}{u}
\newcommand{\vv}{v}

\newcommand{\rk}{\mbox{\rm{rk}}}
\newcommand{\fr}{f}
\newcommand{\eps}{\varepsilon}

\newcommand{\spc}{\mathrm{Spc}}
\newcommand{\I}{\mbox{\rm{im\,}}\!}
\numberwithin{equation}{section}

\begin{document}

\title[Stability of  persistence spaces]{Stability of  persistence spaces of vector-valued continuous functions}

\author{A. Cerri}
\address{Istituto di Matematica Applicata e Tecnologie Informatiche ``Enrico Magenes'', National Council of Research, Via de Marini 6, I-16149 Genova, Italia\newline
ARCES, Universit\`a di
Bologna, via Toffano $2/2$, I-$40135$ Bologna, Italia}
%\curraddr{Department of Mathematics and Statistics,
%Case Western Reserve University, Cleveland, Ohio 43403}
\email{andrea.cerri@ge.imati.cnr.it}
%%    \thanks will become a 1st page footnote.

\author{C. Landi}
\address{Dipartimento di Scienze e Metodi dell'Ingegneria, Universit\`a di Modena e Reggio Emilia, Via Amendola 2, Pad. Morselli, I-42100 Reggio Emilia, Italia\newline ARCES, Universit\`a di Bologna, via Toffano
$2/2$, I-$40135$ Bologna, Italia} 
\email{claudia.landi@unimore.it}

%    General info

\date{} %and, in revised form, June 22, 2001.}



\begin{abstract}
Multidimensional persistence modules do not admit a concise representation analogous to that provided by persistence diagrams for real-valued functions. However, there is no obstruction for multidimensional persistent Betti numbers to admit one. Therefore, it is reasonable to look for a generalization of persistence diagrams concerning those properties that are related only to persistent Betti numbers. In this paper, the {\em persistence space} of a vector-valued continuous function is introduced to generalize the concept of persistence diagram in this sense. The main result is its stability under function perturbations: any change in vector-valued functions implies a not greater change in the Hausdorff distance between their persistence spaces. 
\end{abstract}

\maketitle

\section{Introduction}
Topological data analysis deals with the study of global features of data to extract information about the phenomena that data represent. The persistent homology approach to topological data analysis is based on computing homology groups at different scales to see which features are long-lived and which are short-lived. The basic assumption is that relevant features and structures are the ones that persist longer. 

In classical persistence,  a topological space $X$ is explored through the evolution of the sublevel sets of a real-valued continuous function $f$ defined on $X$. The role of $X$ is to represent the data set, while $f$ is a descriptor of some property which is considered relevant for the analysis. These sublevel sets, being nested by inclusion, produce a filtration of $X$.  Focusing on the occurrence of important topological events along this filtration -- such as the birth and death of connected components, tunnels and voids -- it is possible to obtain a global description of data, which can be formalized via an algebraic structure called a \emph{persistence module} \cite{CaCo*09}. Such information can be encoded in a parameterized version of the Betti numbers, known in the literature as \emph{persistent Betti numbers} \cite{EdHa10}, a \emph{rank invariant} \cite{CaZo09} and -- for the 0$th$ homology -- a \emph{size function} \cite{FrLa01}. The key point is that these descriptors can be represented in a very simple and concise way, by means of multi-sets of points called \emph{persistence diagrams}. Moreover, they are stable with respect to the bottleneck  and  Hausdorff distances, thus implying resistance to noise \cite{CoEdHa07}. Thanks to this property, persistence is a viable option for analyzing data from the topological perspective, as shown, for example, in a number of concrete problems concerning shape comparison and retrieval \cite{BiBa*12,CaZo*05,DiLaMe09}. 
 
A common scenario in applications is to deal with multi-parameter information. The use of vector-valued  functions  enables the study of multi-parameter filtrations, whereas a scalar-valued function only gives  a one-parameter filtration. Therefore, Frosini and Mulazzani \cite{FrMu99} and Carlsson and Zomorodian \cite{CaZo09} proposed \emph{multidimensional persistence} to analyze richer and more complex data. 

A major issue in multidimensional persistence is that, when filtrations depend on multiple  parameters, it is not possible to provide a complete and discrete representation for multidimensional persistence modules analogous to that provided by persistence diagrams for one-dimensional persistence modules \cite{CaZo09}. This theoretical obstruction discouraged so far the introduction of a multidimensional analogue of the persistence diagram. 

One can immediately see that the lack of such an  analogue is a severe drawback for the actual application of multidimensional persistence to the analysis of data. Therefore a natural question we may ask ourselves is the following one: In which other sense may we hope to construct a generalization of a persistence diagram for the multidimensional setting? 

Cohen-Steiner {\em et al.} \cite{CoEdHa07} showed that the persistence diagram satisfies the following important properties (see also \cite{CeDi*12} for the generalization from tame to arbitrary continuous functions):
\begin{itemize}
\item it can be defined via {\em multiplicities} obtained from persistent Betti numbers;
\item it allows to completely reconstruct persistent Betti numbers;
\item it is stable with respect to function perturbations;
\item the coordinates of its off-diagonal points are homological critical values. 
\end{itemize}  

Therefore, it is reasonable to require that a generalization of a persistence diagram for the multidimensional setting satisfies all these properties.  We underline that, because of the aforementioned impossibility result in \cite{CaZo09}, no  generalization of a persistence diagram exists that can achieve the goal of representing completely a persistence module, but only its persistent Betti numbers. For this reason, in this paper we will only study persistent Betti numbers and not persistence modules.

In the present work we introduce a \emph{persistence space} to generalize the notion of a persistence diagram in the  aforementioned sense. More precisely, we define a persistence space as a multiset of points defined via multiplicities. In the one-dimensional case it coincides with a persistence diagram. Moreover, it allows for a complete reconstruction of multidimensional persistent Betti numbers (Multidimensional Representation Theorem~\ref{representation}). These ideas were anticipated in \cite{CeLa13}.

Our main result is the stability of persistence spaces under function perturbations (Stability Theorem~\ref{stabilitySpaces}): the Hausdorff distance between the persistence spaces of two functions $f,g:X\to \R^n$ is never greater than $\max_{x\in X}\max_{1\le i\le n}|f_i(x)-g_i(x)|$.

As a further contribution of this paper we show that the coordinates of the off-diagonal points of a persistence space are multidimensional homological critical values (Theorem~\ref{cpt=>hcv}). 

{\em Outline.} In Section~\ref{background} we review the basics on multidimensional persistent Betti numbers functions and we fix notations. In Section~\ref{spc} we look at discontinuity points of persistent Betti numbers functions in order to define multiplicity of points. Then persistence spaces are introduced and are proven to characterize persistent Betti numbers. We establish the stability result in  Section~\ref{stability}. In Section~\ref{hcv} we show that points of a persistence space have coordinates that are homological critical values. Section~\ref{conclusions} concludes the paper.

\section{Background on persistence}\label{background}
The main reference about multidimensional persistence modules is \cite{CaZo09}. As for multidimensional persistence Betti numbers, we refer the reader to \cite{CeDi*12}. In accordance with the main topic of this paper, in what follows we will stick to the notations and working assumptions adopted in the latter. 

Hereafter, $X$ is a topological space which is assumed to be compact and triangulable, and any function from $X$ to $\R^n$ is supposed to be continuous. When $\R^n$ is viewed as a vector space, its elements are denoted using overarrows. Moreover, in this case, we endow $\R^n$ with the max-norm defined by $\|\vec v\|_{\infty}=\max_i |v_i|$.   

For every $u=(u_1,\dots,u_n),v=(v_1,\dots,v_n)\in\R^n$, we write $u\preceq v$ (resp. $u\prec v$, $u\succ v$, $u\succeq v$) if and only if $u_i\leq v_i$ (resp. $u_i<v_i$, $u_i>v_i$, $u_i\geq v_i$) for all $i=1,\dots,n$. Note that $u\succ v$ is not the negation of $u\preceq v$. 

For every function $f:X\to\R^n$, we denote by $X\langle\fr\preceq u\,\rangle$ the sublevel set $\{x\in X:f(x)\preceq u\}$. We also use the following notations: $D_n^+$ will be the open set $\{(\vu,\vv)\in\R^n\times\R^n:\vu\prec\vv\}$, while $D_n=\{(\vu,\vv)\in D_n^+: \exists j\, \mbox{s.t.}\, \vu_j=\vv_j\}$. $D_n^*$ will denote the set $D_n^+\cup\{(u,\infty):u\in\R^n\}$. Finally, $\overline{D_n^*}=D_n^*\cup D_n$. Points of $D_n^+$ are called {\em proper points}, those of $D_n^*\setminus D_n^+$ are {\em points at infinity}.

For $u\preceq v$, we can consider the inclusion of $X\langle\fr\preceq u\,\rangle$ into $X\langle\fr\preceq v\,\rangle$. This inclusion induces a homomorphism $\iota^{\vu,\vv}_k:\check{H}_k(X\langle\fr\preceq\vu\rangle \rightarrow\check{H}_k(X\langle\fr\preceq\vv\rangle)$, where $\check{H}_k$ denotes the $k$th \v{C}ech homology group for every $k\in\mathbb{Z}$. The image of $\iota^{\vu,\vv}_k$ consists of the $k$-homology classes of cycles ``born'' no later than $\vu$ and ``still alive'' at $\vv$. The use of \v{C}ech homology will shortly be motivated. 

\begin{defi}[Multidimensional persistent homology group]\label{BornDeathClasses}
If $\vu\prec\vv$, the image of $\iota^{\vu,\vv}_k$ is called the {\em
multidimensional $k$th persistent homology group of $(X,\fr)$ at
$(\vu, \vv)$}.
\end{defi}

We assume to work with coefficients in a field $\mathbb{K}$. Hence homology groups are vector spaces, and homomorphisms induced in homology by continuous maps are linear maps. As usual, by the rank of a linear map we mean the dimension of its image. Thus the rank of $\iota^{\vu,\vv}_k$ completely determines  persistent homology groups, leading to the notion of \emph{persistent Betti numbers}.

\begin{defi}[Persistent Betti Numbers]\label{Rank}
The {\em persistent Betti numbers function of $\fr:X\to\R^n$} (briefly PBNs) is the function $\beta_{\fr}:D_n^+\to\N$ defined, for $(u,v)\in D_n^+$,  by
$$
\beta_{\fr}(\vu,\vv)=\rk \,\iota^{\vu,\vv}_k.
$$
\end{defi}

 Obviously, for each $k\in\mathbb{Z}$, we have different PBNs for $f$ (which should be denoted $\beta_{\fr,\,k}$, say) but, for the sake of notational simplicity, we omit adding any reference to $k$. This will also apply to the notations used for other concepts in this paper, such as multiplicities. Among the properties of PBNs, it is worth mentioning those useful in the rest of the paper.

\begin{prop}[Finiteness]\label{RemFiniteness}
For every $(\vu,\vv)\in D_n^+$, $\beta_f(\vu,\vv)<+\infty$.
\end{prop}

\begin{prop}[Monotonicity]\label{Monotonicity}
As an integer-valued function in $(\vu,\vv)\in D_n^+$, $\beta_{\fr}$ is non-decreasing in $\vu$ and non-increasing  in $\vv$.
\end{prop}

\begin{prop}[Right-Continuity]\label{Right}
As a function in $(\vu,\vv)\in D_n^+$, $\beta_{\fr}$ is right-continuous with respect to both $u$ and $v$, that is, $\lim_{\vu\to\bar\vu, \vu\succeq \bar\vu}\beta_f(\vu,\vv) = \beta_f(\bar \vu,\vv) $ and $\lim_{\vv\to\bar\vv, \vv\succeq\bar\vv}\beta_f(\vu,\vv) = \beta_f(\vu,\bar\vv)$.
\end{prop}

The latter property, proved in \cite{CeDi*12} for $n=1$ but  valid also for $n>1$, justifies the use of \v{C}ech theory. The proof is based of the continuity axiom of   \v{C}ech homology (the reader can refer to \cite{EiSt} for details). Using the right-continuity property, in \cite{CeDi*12} it has been proved that for $n=1$ a discrete set of points, called a  {\em persistence diagram}, completely describes  persistent Betti numbers,  without requiring tameness of functions.

\section{Persistence space}\label{spc}

The aim of this section is to introduce 
persistence spaces by analogy with persistence diagrams. 
In order to do this, we preliminarily study the behavior of discontinuity points of PBNs. In particular, we will prove some results about the propagation of discontinuities of PBNs (Corollary~\ref{MultiPropagation}) and about local constancy of PBNs (Proposition~\ref{Weps} and~\ref{Veps}). These facts will be used to introduce the notion of multiplicity of a point (either proper or at infinity). Points of a persistence space will be exactly those with a positive multiplicity. 

The main result of this section is that a persistence space is sufficient to reconstruct the underlying PBNs (Representation Theorem~\ref{representation}), in analogy with the one-dimensional framework (cf. the $k$-Triangle Lemma in \cite{CoEdHa07} and the Representation Theorem 3.11 in \cite{CeDi*12}).

\subsection{PBNs and discontinuities}
We recall that PBNs are  functions from $D_n^+$ to $\N$. Being integer-valued functions, PBNs have jump discontinuities (unless they are identically zero).  Precisely, discontinuity points are  points $(\vu,\vv)$ of $D_n^+$ such that in every neighborhood of $(\vu,\vv)$ in $D_n^+$ there is a point $(\vu',\vv')$ with  $\beta_f(\vu,\vv)\ne \beta_f(\vu',\vv')$.  We now study the behavior of discontinuities of PBNs. We start with some lemmas. Lemma~\ref{Jump} is analogous to \cite[Lemma 1]{FrLa01}, Lemma~\ref{LemmaDisc} is analogous to \cite[Lemma 2]{FrLa01}. Corollary~\ref{MultiPropagation} is analogous to \cite[Cor. 1]{FrLa01}. 
 
It is convenient to introduce the following notations. For $v\in\R^n$, $\beta_{f}(\cdot,v):\R^n\to\mathbb{N}$ denotes the function taking each $n$-tuple $u\prec v$ to the number $\beta_{f}(u,v)$. Analogous meaning will be given to $\beta_{f}(u,\cdot)$. For every $\bar{\vu}=(\bar{\vu}_1,\dots,\bar{\vu}_n)\in\R^n$,
we denote by $\R^n_{\pm}(\bar{\vu})$ the subset of $\R^n$ given by $\{\vu\in\R^n: \vu\prec\bar{\vu}\ \vee\ \vu\succ\bar{\vu}\}$. In particular, note that for a point $u=(u_1,\dots,u_n)$ in $\R^n_{\pm}(\bar{\vu})$ it holds that $u_i\neq\bar{\vu}_i$ for every $i=1,\dots,n$.

\begin{lem}[Multidimensional Jump Monotonicity]\label{Jump}
Let $\vu^{_1},\vu^{_2},\vv^{_1},\vv^{_2}\in\R^n$. If
$\vu^{_1}\preceq\vu^{_2}\prec\vv^{_1}\preceq\vv^{_2}$, then
$$\beta_{f}(\vu^{_2},\vv^{_1})-\beta_{f}(\vu^{_1},\vv^{_1})\ge
\beta_{f}(\vu^{_2},\vv^{_2})-\beta_{f}(\vu^{_1},\vv^{_2}).$$
\end{lem}
\begin{proof}
The value $\beta_{f}(\vu^{_2},\vv^{_1})-\beta_{f}(\vu^{_1},\vv^{_1})$
(resp. $\beta_{f}(\vu^{_2},\vv^{_2})-\beta_{f}(\vu^{_1},\vv^{_2})$)
represents, by Definition~\ref{BornDeathClasses}, the number of linearly independent homology classes of cycles ``born''
between $\vu^{_1}$ and $\vu^{_2}$, and ``still alive'' at $\vv^{_1}$ (resp. $\vv^{_2}$). Therefore, our claim follows by observing that the number of linearly independent homology classes born between $\vu^{_1}$ and $\vu^{_2}$ and still alive at $\vv^{_1}$ is certainly not smaller than the number of those still alive at $\vv^{_2}$.
\end{proof}

\begin{lem}\label{LemmaDisc}
In $D_n^+$, any  neighborhood of a discontinuity point of $\beta_f$ contains a point $(\vu,\vv)$ with $\beta_{f}(\cdot,\vv)$ discontinuous at $\vu$, or $\beta_{f}(\vu,\cdot)$ discontinuous at $\vv$. 
\end{lem}

\begin{proof}
Every  neighborhood  of $p$ in $ D_n^+$ contains an open hyper-cube $Q$ centered at $p$. If $p$ is a discontinuity point of $\beta_{f}$, there is a point $q\in Q$ with $\beta_f(p)\neq\beta_f(q)$. We can connect $p$ and $q$ by a path entirely contained in $Q$ made of segments such that either the $n$-tuple $u$ or the $n$-tuple $v$ is constant for all points $(u,v)$ of each such segment. $\beta_f$ cannot be constant along this path. This proves the claim.
\end{proof}

\begin{lem}\label{lemma0}

For every $(\bar{\vu},\bar{\vv})\in D_n^+$, the following statements hold:

\begin{enumerate}
\item[$(i)$] If $\bar{\vu}$ is a discontinuity point of $\beta_{f}(\cdot,\bar{\vv})$, then, for every real number $\varepsilon>0$, there is a point $u\in\R^n_{\pm}(\bar{\vu})$ such that $\|\vu-\bar\vu\|_\infty<\varepsilon$ and $\beta_f(\vu,\bar\vv)\ne \beta_f(\bar\vu,\bar\vv)$;
\item[$(ii)$] If $\bar{\vv}$ is a discontinuity point of $\beta_{f}(\bar{\vu},\cdot)$,  then, for every real number $\varepsilon>0$, there is a point $v\in\R^n_{\pm}(\bar{\vv})$ such that $\|\vv-\bar\vv\|_\infty<\varepsilon$ and $\beta_f(\bar\vu,\vv)\ne\beta_f(\bar\vu,\bar\vv)$.
\end{enumerate}
\end{lem}

\begin{proof}
$(i)$ If $\bar{\vu}$ is a discontinuity point of $\beta_{f}(\cdot,\bar{\vv})$, then, for every $\varepsilon>0$, there is a point $u'\in\R^n$ such that $\|\vu'-\bar\vu\|_\infty<\varepsilon$ and $\beta_f(\vu',\bar\vv)\neq\beta_f(\bar\vu,\bar\vv)$. Let us consider the case when $\beta_f(\vu',\bar\vv)<\beta_f(\bar\vu,\bar\vv)$. If $\vu'\notin \R^n_{\pm}(\bar{\vu})$, we take a point $\vu\in\R^n_{\pm}(\bar{\vu})$ such that $\vu\preceq\vu'$ and $\|\vu-\bar\vu\|_\infty<\varepsilon$. By the monotonicity of PBNs (Proposition~\ref{Monotonicity}), $\beta_f(\vu,\bar v)\leq\beta_f(\vu',\bar\vv)$. Hence, $\beta_f(\vu,\bar\vv) < \beta_f(\bar\vu,\bar\vv)$, yielding the claim. The case when  $\beta_f(\vu',\bar\vv)>\beta_f(\bar\vu,\bar\vv)$ can be handled in much the same way. 

$(ii)$ The proof is analogous.
\end{proof}

\begin{cor}\label{MultiPropagation}
For every $(\bar{\vu},\bar{\vv})\in D_n^+$, the following statements hold:
\begin{enumerate}
\item[$(i)$] If $\bar{\vu}$ is a discontinuity point of $\beta_{f}(\cdot,\bar{\vv})$, then it is a discontinuity point of  $\beta_{f}(\cdot,\vv)$  for every $\bar\vu\prec\vv\preceq\bar{\vv}$; 
\item[$(ii)$] If $\bar{\vv}$ is a discontinuity point of  $\beta_{f}(\bar{\vu},\cdot)$, then it is a discontinuity point of $\beta_{f}(\vu,\cdot)$  for every $\bar\vu\preceq\vu\prec\bar{\vv}$.
\end{enumerate}
\end{cor}

\begin{proof}
We shall confine ourselves to prove only statement $(i)$. Indeed,
proving statement $(ii)$ is completely analogous.

Contrary to our claim assume that, for some $\vv$ with $\bar \vu \prec \vv \preceq \bar \vv$, $\beta_f(\cdot, \vv)$ is continuous at $\bar \vu$. Then $\lim_{\vu\to\bar\vu, \vu\succeq \bar\vu}\beta_f(\vu,\vv) - \beta_f(\bar \vu,\vv) = 0$.
Hence Lemma~\ref{Jump} together with the fact that PBNs are non-decreasing in $\vu$ (Proposition~\ref{Monotonicity}) imply that $\lim_{\vu\to\bar\vu, \vu\succeq \bar\vu}\beta_f(\vu,\bar\vv) - \beta_f(\bar \vu,\bar\vv) = 0$. Analogously, $\lim_{\vu\to\bar\vu, \vu\preceq \bar\vu}\left( \beta_f(\bar \vu,\bar\vv)-\beta_f(\vu,\bar\vv) \right)= 0$. Hence, for some sufficiently small $\varepsilon>0$, and for every $\vu\in\R^n_{\pm}(\bar{\vu})$ such that $\|\vu-\bar\vu\|_\infty<\varepsilon$, recalling that $\beta_f$ is integer-valued, we have $\beta_f(\vu,\bar\vv)= \beta_f(\bar\vu,\bar\vv)$. By Lemma~\ref{lemma0}$(i)$ this implies that $\bar\vu$ cannot be a discontinuity point of $\beta_f(\cdot,\bar v)$.
\end{proof}

The next two propositions (analogous to \cite[Prop. 6]{FrLa01} and \cite[Prop. 7]{FrLa01}, respectively) give some constraints on the
presence of discontinuity points of PBNs.  

\begin{prop}\label{Weps}
Let $\bar{p}=(\bar{\vu},\bar{\vv})$ be a proper point of
$D_n^+$. Then, a real number $\varepsilon>0$ exists, such that the
open set
$$W_{\varepsilon}(\bar{p})=\{(\vu,\vv)\in\R^n_{\pm}(\bar{\vu})\times\R^n_{\pm}(\bar{\vv}):\|\vu-\bar{\vu}\|_{\infty}<\varepsilon,\, \|\vv-\bar{\vv}\|_{\infty}<\varepsilon\}$$%,\, \vu_i\neq\bar{\vu}_i,\, \vv_i\neq\bar{\vv}_i,\, i=1,\dots,n\}$$
is a subset of $ D_n^+$, and does not contain any discontinuity point of $\beta_{f}$.
\end{prop}

\begin{proof}
Obviously, there always exists a sufficiently small $\varepsilon>0$ such that $W_{\varepsilon}(\bar{p})\subseteq D_n^+$ (see Figure~\ref{fig:Weps} to visualize this). Let us now fix $N\in\N$ such that $1/N<\varepsilon$, and suppose, contrary to our assertion, that for every $j\geq N$ a discontinuity point $p^{_j}$ of $\beta_{f}$ exists in $W_{1/j}(\bar{p})$. We want to prove that this contradicts the finiteness of $\beta_{f}$ (cf. Proposition~\ref{RemFiniteness}).

\begin{figure}
\psfrag{u}{$\bar u$}\psfrag{v}{$\bar v$}\psfrag{e}{$\varepsilon$}
\includegraphics[width=0.5\textwidth]{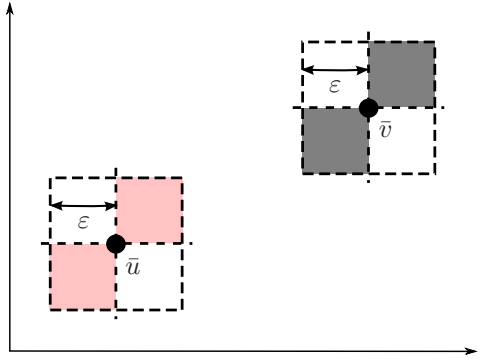}
\caption{Graphical representation of $W_{\varepsilon}(\bar{p})\subseteq D_n^+$, with $n=2$, corresponding to the Cartesian product of the (light) pink area with the (dark) gray one.}
\label{fig:Weps}
\end{figure}

We set $\bar{\vu}_{+1/N}=(\bar{\vu}_1+1/N,\dots,\bar{u}_n+1/N)$, $\bar\vv_{-1/N}=(\bar{\vv}_1-1/N,\dots,\bar{\vv}_n-1/N)$.

By the previous Lemma~\ref{LemmaDisc} we know that arbitrarily close to each $p^{_j}$ there is a point $q^{_j}=(u^{_j},v^{_j})$ such that either $u^{_j}$ is a discontinuity point of $\beta_{f}(\cdot,\vv^{_j})$,
or $\vv^{_j}$ is a discontinuity point of $\beta_{f}(\vu^{_j},\cdot)$. Therefore, possibly by extracting a subsequence from $(q^{_j})_{j\geq N}$, we can assume that each $u^{_j}$ is a discontinuity point of $\beta_{f}(\cdot,v^{_j})$. Since $q^{_j}$ can be taken arbitrarily close to $p^{_j}$, we can also assume that $\vu^j\preceq\bar{\vu}_{+1/N}$ and, possibly by considering again a subsequence, that either $\vu^{_j}\preceq\vu^{_{j+1}}\preceq\bar \vu$ for every $j\geq N$, or
$\bar\vu\preceq\vu^{_{j+1}}\preceq\vu^{_{j}}$ for every $j\geq N$.

By Corollary~\ref{MultiPropagation}$(i)$ it holds that $\vu^{_j}$ is a
discontinuity point of $\beta_{f}(\cdot,\vv)$, for every  $\vu^{_j}\prec \vv\preceq\vv^{_j}$. In particular, for $v=\bar v_{-1/N}$, $\beta_f(\cdot,v)$ has an integer jump at each $\vu^j$. Moreover, we have also that $\vu^{_j}\preceq\bar{\vu}_{+1/N}$ for every $j\geq N$. Therefore, since $\beta_{f}(\cdot,\bar \vv_{-1/N})$ is non-decreasing (cf. Proposition~\ref{Monotonicity}), and recalling that either $\vu^{_j}\preceq\vu^{_{j+1}}$ for every $j\geq N$, or $\vu^{_{j+1}}\preceq\vu^{_{j}}$ for every $j\geq N$, we deduce that $\beta_{f}(\bar{\vu}_{+1/N},\bar{\vv}_{-1/N})=+\infty$, thus contradicting the finiteness of $\beta_{f}$.
\end{proof}

\begin{prop}\label{Veps}
Let $\bar{p}=(\bar\vu,\infty)$ be a point at infinity of $D^*_n$. Then, a real number
$\varepsilon>0$ exists, such that the open set $$V_{\varepsilon}(\bar{p})=\{(\vu,\vv)\in\R^n_{\pm}(\bar{\vu})\times\R^n:\|\vu-\bar{\vu}\|_{\infty}<\varepsilon,\, \vv_i>\frac{1}{\varepsilon},\, i=1,\dots,n\}$$ is a subset of $ D_n^+$, and does not contain any discontinuity point of $\beta_{f}$.
\end{prop}
\begin{proof}
First of all, let us observe that $V_{\varepsilon}(\bar{p})\subseteq D_n^+$ for $\varepsilon$ sufficiently close to 0. Next, we fix $N\in\N$ such that $1/N<\varepsilon$, and suppose, contrary to our assertion, that for every $j\leq N$ a discontinuity point $p^{_j}$ of $\beta_{f}$ exists in $V_{1/j}(\bar{p})$. We want to prove that such an assumption leads to contradict the finiteness of $\beta_{f}$ (cf. Proposition~\ref{RemFiniteness}).

By the previous Lemma~\ref{LemmaDisc} we know that arbitrarily close to each $p^{_j}$ there is a point $q^{_j}=(u^{_j},v^{_j})$ such that either $u^{_j}$ is a discontinuity point of $\beta_{f}(\cdot,\vv^{_j})$, or $\vv^{_j}$ is a discontinuity point of $\beta_{f}(\vu^{_j},\cdot)$. It is not restrictive to assume $N$ sufficiently large such that $\max_{x\in X}\|f(x)\|_{\infty}\leq N$. 
%It follows that, for every $j\geq N$, each $p^{_j}=(\vu^{_j},\vv^{_j})$ is such that $X\langle f\preceq\vv^{_j}\rangle=X$. 
Since $q^{_j}$ can be taken arbitrarily close to $p^{_j}$, we can assume that  $q^{_j}\in V_{1/j}(\bar p)$, implying that $v^{_j}\succ j$. Hence, $X\langle f\preceq\vv^{_j}\rangle=X$ for every $j\geq N$. Therefore, $\vv^{_j}$ is not a discontinuity point of $\beta_{f}(\vu^{_j},\cdot)$, and
hence $\vu^{_j}$ is a discontinuity point of $\beta_{f}(\cdot,\vv^{_j})$. Thus, by Corollary~\ref{MultiPropagation}$(i)$ we have that, for every $j\geq N$, $\vu^{_j}$ is a discontinuity point of $\beta_{f}(\cdot,\vv)$, with $\vu^{_j}\prec \vv\preceq\vv^{_j}$. Before going on note that, possibly by considering a subsequence of $(q^{_j})_{j\ge N}$, we can assume that either $\vu^{_j}\preceq\vu^{_{j+1}}$ for every $j\geq N$, or $\vu^{_{j+1}}\preceq\vu^{_{j}}$ for every $j\geq N$.

We now set $\bar{\vu}_{+1/N}=(\bar{\vu}_1+1/N,\dots,\bar{u}_n+1/N)$, $\bar{v}=(N,\dots,N)$, and consider the function $\beta_{f}(\cdot,\bar{v})$. According to the previous considerations, such a function should have an infinite number of integer jumps. Indeed, for every $j\geq N$ we have $\bar v\prec v^{_j}$ and hence $\vu^{_j}$ is a discontinuity point of
$\beta_{f}(\cdot,\bar v)$. Moreover, we have also that $u^{_j}\preceq\bar{\vu}_{+1/N}$ for every $j\geq N$. Therefore, since $\beta_{f}(\vu,\vv)$ is non-decreasing in the variable $u$ (cf. Proposition~\ref{Monotonicity}), and recalling that either $\vu^{_j}\preceq\vu^{_{j+1}}$ for every $j\geq N$, or $\vu^{_{j+1}}\preceq\vu^{_{j}}$ for every $j\geq N$, it should be
that $\beta_{f}(\bar{\vu}_{+1/N},\bar v)=+\infty$,
thus contradicting the finiteness of $\beta_{f}$ (cf. Proposition~\ref{RemFiniteness}).
\end{proof}

\subsection{Persistence Space and Representation Theorem}
In this section we introduce persistence spaces of vector-valued functions, in analogy to persistence diagrams of scalar-valued functions. To do this, we first generalize the notion of multiplicity of a point to the multidimensional setting. Then we introduce cornerpoints as points with a non-zero multiplicity, and persistence spaces as multisets of cornerpoints. Finally, we show that the PBNs of a continuous function can be reconstructed from multiplicities of points. 

We begin with defining the multiplicity of a proper point and the notion of a proper cornerpoint. For every $(\vu,\vv)\in D_n^+$ and $\vec e\in\R^n$ with $\vec e\succ 0$ and $u+\vec e\prec v-\vec e$, we consider the number  
{\setlength\arraycolsep{2pt}
\begin{equation}\label{MMFormula1}
\begin{array}{cccc}
\mu_{f,\vec e\,}(\vu,\vv)=\beta_f(\vu+\vec e ,\vv-\vec e)&-&\beta_f(\vu-\vec e ,\vv-\vec e)+&\\
&-&\beta_f(\vu+\vec e,\vv+\vec e)+&\beta_f(\vu-\vec e ,\vv+\vec e).
\end{array}
\end{equation}
The computation of $\mu_{f,\vec e\,}(\vu,\vv)$ is illustrated in Figure~\ref{fig:mu_shape}.

\begin{figure}
\psfrag{u}{$u$}\psfrag{u+}{$u+\vec e$}\psfrag{v}{$v$}\psfrag{v+}{$v+\vec e$}\psfrag{u-}{$u-\vec e$}\psfrag{v-}{$v-\vec e$}\psfrag{e}{$\vec e$}
\includegraphics[width=0.5\textwidth]{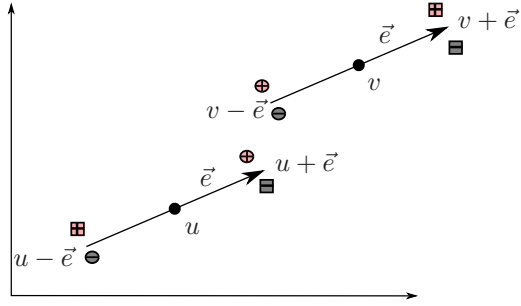}
\caption{The computation of $\mu_{f,\vec e\,}(\vu,\vv)$ involves the algebraic sum of the values that $\beta_f$ takes at the four points  $(\vu+\vec e ,\vv-\vec e)$, $(\vu-\vec e ,\vv-\vec e)$, $(\vu+\vec e,\vv+\vec e)$, $(\vu-\vec e ,\vv+\vec e)$. In this picture the pairs of coordinates of each point  are shape- and color-coded, and the plus sign in the sum is represented by $\oplus$ or $\boxplus$, the minus sign by  $\ominus$ or $\boxminus$.}
\label{fig:mu_shape}
\end{figure}

PBNs being  integer-valued functions (Proposition~\ref{RemFiniteness}), we have that $\mu_{f,\vec e\,}(\vu,\vv)$ is an integer number, and by Lemma~\ref{Jump} it is non-negative. Once again by Lemma~\ref{Jump}, if $ 0\prec\vec e\preceq\vec\eta$, then
\begin{eqnarray*}
\quad\quad\begin{array}{ccc} 
\beta_{f}(\vu+\vec\eta,\vv-\vec\eta) - \beta_{f}(\vu-\vec\eta,\vv-\vec\eta) & \geq\beta_{f}(\vu+\vec\eta,\vv-\vec e\,) - \beta_{f}(\vu-\vec\eta,\vv-\vec e\,),\\
\beta_{f}(\vu+\vec\eta,\vv+\vec\eta) - \beta_{f}(\vu-\vec\eta,\vv+\vec\eta) & \leq\beta_{f}(\vu+\vec\eta,\vv+\vec e\,) - \beta_{f}(\vu-\vec\eta,\vv+\vec e\,),
\end{array}\\
%\end{equation}
%\begin{equation}
\quad\quad\begin{array}{ccc} 
\beta_{f}(\vu+\vec\eta,\vv-\vec e\,) - \beta_{f}(\vu+\vec\eta,\vv+\vec e\,) & \geq\beta_{f}(\vu+\vec e,\vv-\vec e\,) - \beta_{f}(\vu+\vec e,\vv+\vec e\,),\\
\beta_{f}(\vu-\vec\eta,\vv-\vec e\,) - \beta_{f}(\vu-\vec\eta,\vv+\vec e\,) & \leq\beta_{f}(\vu-\vec e,\vv-\vec e\,) - \beta_{f}(\vu-\vec e,\vv+\vec e\,).
\end{array}
\end{eqnarray*}
These inequalities easily imply that the sum defining $\mu_{f,\vec e\,} (\vu,\vv)$ is non-decreasing in $\vec e$. Moreover, by Proposition~\ref{Weps}, each term in that sum is constant on the set of those $\vec e\in\R^n$ for which $\vec e\succ 0$ and $\left\|\vec e\right\|_\infty$ is sufficiently close to 0. These remarks justify the following definition. 

\begin{defi}\label{multiplicityp}
For every $p=(\vu,\vv)\in D_n^+$, the {\em multiplicity } $\mu_{f}(p)$ is the finite, non-negative integer number defined by setting
{\setlength\arraycolsep{2pt}
$$%\begin{equation}\label{proper-cpt}
\mu_{f}(p)=\min_{\newatop{\vec e\succ 0}{\vu+\vec e\prec \vv-\vec e}}\mu_{f,\vec e\,}(\vu,\vv)
$$%\end{equation}
with $\vec e\in\R^n$. When  $\mu_{f} (p)$ is strictly positive, $p$ is said to be a {\em proper cornerpoint} of $\beta_{f}$.}
\end{defi}

In plain words, a proper cornerpoint captures a topological feature with bounded persistence. Persistence can be defined as follows.
\begin{defi}\label{persistence}
The {\em persistence} of a proper cornerpoint $p=(u,v)$ is given by
$$\mathrm{pers}(p)=\min_{i=1,\ldots,n}v_i-u_i.$$
\end{defi}  
The motivation behind this definition of persistence is that $\min_{i=1,\ldots,n}v_i-u_i$ is directly proportional to the distance of $(u,v)$ to $D_n$, as explained at the beginning of Section~\ref{stability}. Therefore, it gives a measure of the amount of perturbation needed to cancel a proper cornerpoint. Interestingly, $\min_{i=1,\ldots,n}v_i-u_i$ is also equal to the one-dimensional persistence in the filtration along the line passing through $u$ and $v$, as treated in \cite[Theorem~4.2]{CeDi*12}.

We now similarly define the multiplicity of a point at infinity and the notion of cornerpoint at infinity. For every $(\vu,\vv)\in D_n^+$ and $\vec e\in\R^n$ with $\vec e\succ 0$ and $u+\vec e\prec v$, we consider the number  
\begin{equation}\label{MMFormula2}
\begin{array}{cccc}
\mu_{f,\vec e\,}^{_{\infty}}(\vu,\vv)=\beta_{f}(\vu+\vec e,\vv)&-& \beta_{f}(\vu-\vec e,\vv).
\end{array}
\end{equation}
The computation of $\mu_{f,\vec e\,}^{_{\infty}}(\vu,\vv)$ is illustrated in Figure~\ref{fig:mu_infty}.

\begin{figure}
\psfrag{u}{$u$}\psfrag{u+}{$u+\vec e$}\psfrag{v}{$v$}\psfrag{u-}{$u-\vec e$}\psfrag{e}{$\vec e$}
\includegraphics[width=0.3\textwidth]{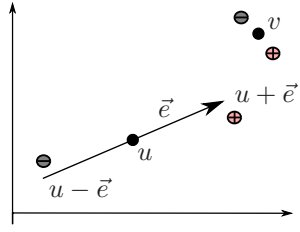}
\caption{The computation of $\mu_{f,\vec e\,}^{_{\infty}}(\vu,\vv)$ involves the algebraic sum of the values that $\beta_f$ takes at the points  $(\vu+\vec e ,\vv)$, $(\vu-\vec e ,\vv)$. In this picture the pairs of coordinates of each point  are color-coded, and the sign in the sum is represented by $\oplus$ or $\ominus$.}
\label{fig:mu_infty}
\end{figure}

By Proposition~\ref{RemFiniteness}, $\mu_{f,\vec e\,}^{_{\infty}}(\vu,\vv)$ is an integer number, and by Proposition~\ref{Monotonicity} we know that it is non-negative. Lemma~\ref{Jump} easily implies that it is non-decreasing in $\vec e$ and non-increasing in $\vec v$. Moreover, by Proposition~\ref{Veps}, each term of the sum defining $\mu_{f,\vec e\,}^{_{\infty}}(\vu,\vv)$ is constant on the set of those $\vec e,v \in\R^n$ for which $\vec e\succ 0$, $\left\|\vec e\right\|_\infty$ is sufficiently close to 0, and $v_i>1/\|\vec e\|_\infty$, for $i=1,\dots,n$. These remarks justify the following definition.}

\begin{defi}\label{multiplicityr}
For every $p=(u,\infty)\in D_n^*$, the {\em multiplicity } $\mu_{f}(p)$ is the finite, non-negative integer number defined by setting

$$%\begin{eqnarray}\label{cpt-at-infty}
\mu_{f}(p)=\min_{\newatop{\vec e\succ 0}{v: u+\vec e\prec v}} \mu_{f,\vec e\,}^{_{\infty}}(\vu,\vv).
$$%\end{eqnarray}
When $\mu_{f}(p)$ is strictly positive,  $p$ is said to be a {\em cornerpoint at infinity} of $\beta_{f}$.
\end{defi}

\begin{rem}\label{multiplicityrRem0}
For any $f:X\to\R^n$, $\beta_f$ admits at least a cornerpoint at infinity in the homology degree $k$ if and only if $\check{H}_k(X)\neq 0$.
\end{rem}

Basically, a cornerpoint at infinity captures a topological feature that is going to live forever, i.e. it is an essential topological feature of $X$. In other words, cornerpoints at infinity correspond to features with unbounded persistence.

\begin{rem}\label{multiplicityrRem1}
For $n=1$, Definitions~\ref{multiplicityp} and \ref{multiplicityr} coincide with the  definitions of multiplicity of proper points and points at infinity, respectively,  used to define persistence diagrams.
\end{rem}

Having extended the notion of multiplicity to the multidimensional setting, the definition of a persistence space is now completely analogous to the one of a persistence diagram for real-valued functions.

\begin{defi}[Persistence Space]\label{persSpace}
The persistence space $\spc(f)$ is the multiset of all points $p\in D^*_n$ such that $\mu_f(p)>0$, counted with their multiplicity, union the points of $D_n$, counted with infinite multiplicity.
\end{defi}

Persistence spaces can be reasonably thought as the analogue, in the case of a vector-valued function, of persistence diagrams. Indeed, similarly to the one-dimensional case, a persistence space is completely and uniquely determined by the corresponding persistent Betti numbers. Moreover, even in the multi-parameter situation the converse is true as well, since it is possible to prove the following Multidimensional Representation Theorem. In what follows, $\langle\vec e\,\rangle$ denotes the line in $\R^n$ spanned by $\vec e$.

\begin{thm}[Multidimensional Representation Theorem]\label{representation}
For every $(\bar{\vu},\bar{\vv})\in D_n^+$ and every $\vec e\succ 0$, it holds that
\begin{equation}\label{mainFormula}
\beta_{f}(\bar{\vu},\bar{\vv})=\sum_{\newatop{\vu\preceq\bar{\vu},\,v\succ\bar v}{\bar{\vu}-\vu,\vv-\bar{\vv}\in
\langle\vec e\,\rangle}}
\mu_{f} (\vu,\vv)+\sum_{\newatop{\vu\preceq\bar{\vu}}{\bar{\vu}-\vu\in
\langle\vec e\,\rangle}}\mu_{f}(\vu,\infty).
\end{equation}
\end{thm}

\begin{proof}
We have seen that, for every $(u,v)\in D_n^+$, a positive real number  $\varepsilon=\varepsilon(u,v)$ sufficiently small exists, for which
\begin{equation}\label{partialMultiplicity}
\mu_f(u,v)=\mu_{f,\varepsilon\vec e\,}(u,v). 
\end{equation}
As for points at infinity, for every $(u,\infty)\in D^*_n$, we can choose a positive real number $\varepsilon=\varepsilon(u)$ sufficiently small such that, setting $v^{_{\varepsilon}}=(\varepsilon^{_{-1}},\dots,\varepsilon^{_{-1}})$,  we have  
\begin{equation}\label{partialMultiplicityLine}
\mu_f(u,\infty)=\mu_{f,\varepsilon\vec e\,}^{_{\infty}}(u,v^{_{\varepsilon}}).
\end{equation}
Thus, by (\ref{partialMultiplicity}) and (\ref{partialMultiplicityLine}), we  get
\begin{eqnarray*}
\sum_{\newatop{\vu\preceq\bar{\vu},\,v\succ\bar v}{\bar{\vu}-\vu,\vv-\bar{\vv}\in
\langle\vec e\,\rangle}}
\mu_{f} (\vu,\vv) & = &\sum_{\newatop{\vu\preceq\bar{\vu},\,v\succ\bar v}{\bar{\vu}-\vu,\vv-\bar{\vv}\in
\langle\vec e\,\rangle}}\mu_{f,\varepsilon\vec e\,}(\vu,\vv),\label{sumMultiplicity}\\
\sum_{\newatop{\vu\preceq\bar{\vu}}{\bar{\vu}-\vu\in
\langle\vec e\,\rangle}}
\mu_{f}(\vu,\infty) & = &\sum_{\newatop{\vu\preceq\bar{\vu}}{\bar{\vu}-\vu\in
\langle\vec e\,\rangle}}\mu_{f,\varepsilon\vec e\,}^{_{\infty}}(u,v^{_{\varepsilon}}).\label{sumMultiplicity2}
\end{eqnarray*}
Now, by the finiteness and the monotonicity of PBNs, at most finitely many discontinuity points of $\beta_f(\cdot,\bar v)$ exist, say $u^{_{1}},\dots,u^{_{p}}\in\R^n$, such that $u^{_{i}}\preceq\bar u$ and $\bar u-u^{_{i}}\in\langle\vec e\,\rangle$ for all $i=1,\dots,p$. Without loss of generality, we can assume that $u^{_1}\prec\dots\prec u^{_p}$. Analogously, let $v^{{_1}}\prec\dots\prec v^{_{q}}$ be the discontinuity points of $\beta_f(\bar u,\cdot)$ for which $v^{_{j}}\succ\bar v$ and $v^{_{j}}-\bar v\in\langle\vec e\,\rangle$ for all $j=1,\dots,q$. In conclusion, we have $u^{_1}\prec\dots\prec u^{_p}\preceq\bar u\prec\bar v\prec v^{_1}\prec\dots\prec v^{_q}$.

Note that, for every $v\succ\bar v$ with $v-\bar v\in\langle\vec e\,\rangle$, the restriction of $\beta_f(\cdot,v)$ to the set $U=\{u\in\R^n\,|\,u\preceq\bar u,\ \bar u-u\in\langle\vec e\,\rangle\}$ is continuous for all $u\neq u^{_{i}}$. Indeed, suppose to the contrary that, for a given $\hat v\succ\bar v$ with  
$\hat v-\bar v\in\langle\vec e\,\rangle$, the function $\beta_f(\cdot,\hat v)$
restricted to the above set is discontinuous at a point $\hat u\ne u^i$. Then, by Corollary~\ref{MultiPropagation} we would have that $\beta_f(\cdot,v)$ is discontinuous at $\hat u$ for every $v$ with $\hat u\prec v\preceq\hat v$. In particular, $\hat u$ would be a discontinuity point of $\beta_f(\cdot,\bar v)$, against the assumption that $u^{_{1}},\dots,u^{_{p}}$ are the only discontinuity points.

In much the same way we can show that, for any $u\preceq\bar u$ with $\bar u-u\in\langle\vec e\,\rangle$, the restriction of $\beta_f(u,\cdot)$ to $V=\{v\in\R^n\,|\,v\succ\bar v,\ v-\bar v\in\langle\vec e\,\rangle\}$ is continuous for all $v\neq v^{_{j}}$.  

These remarks imply that, for $u\in U\setminus\{u^{_{1}},\dots,u^{_{p}}\}$, $v\in V\setminus \{v^{_{1}},\dots,v^{_{q}}\}$, and for any real $\varepsilon>0$ sufficiently small, $\mu_{f,\varepsilon\vec e\,}(u,v)=0$. Indeed, $\beta_f$ is integer-valued so that, by recalling equality (\ref{MMFormula1}), the sum defining $\mu_{f,\varepsilon\vec e\,}(u,v)$ for any such $u$ and $v$ is over constant terms. Similarly, by equality (\ref{MMFormula2}) it holds that $\mu_{f,\varepsilon\vec e\,}^{_{\infty}}(u,v^{_{\varepsilon}})=0$. Therefore we have 
{\setlength\arraycolsep{2pt}
\begin{eqnarray*}\label{mainFormulaTer}
\sum_{\newatop{\vu\preceq\bar{\vu},\,v\succ\bar v}{\bar{\vu}-\vu,\vv-\bar{\vv}\in
\langle\vec e\,\rangle}}\mu_{f,\varepsilon\vec e\,}(\vu,\vv)&=&\sum_{\newatop{i=1,\dots,p}{j=1,\dots,q}}\mu_{f,\varepsilon\vec e\,}(u^{_{i}},v^{_{j}}),\\
\sum_{\newatop{\vu\preceq\bar{\vu}}{\bar{\vu}-\vu\in
\langle\vec e\,\rangle}}\mu_{f,\varepsilon\vec e\,}^{_{\infty}}(u,v^{_{\varepsilon}})&=&\sum_{i=1,\dots,p}\mu_{f,\varepsilon\vec e\,}^{_{\infty}}(u^{_{i}},v^{_{\varepsilon}}).
\end{eqnarray*}
Before going on, note that the following facts hold for every $u\in U$ and $v\in V$:
\begin{description}
\item[(F1)] If $u^{_{i}}\preceq u\prec u^{_{i+1}}$ and $v^{_{j}}\preceq v\prec v^{_{j+1}}$, then $\beta_f(u,v)=\beta_f(u^{_i},v^{_j})$. Indeed, $\beta_f$ is integer-valued and right-continuous (Propositions~\ref{RemFiniteness} and~\ref{Right}). Moreover, $\beta_f(\cdot,v)$ and $\beta_f(u,\cdot)$ are continuous in $U\setminus\{u^{_{1}},\dots,u^{_{p}}\}$ and $V\setminus\{v^{_{1}},\dots,v^{_{q}}\}$, respectively; 
\item[(F2)] If $u\prec u^{_{1}}$ then $\beta_f(u,v)=0$;
\item[(F3)] If $v\succeq v^{_q}$ then $\beta_f(u,v)=\beta_f(u,v^{_q})$.
%\item[(F4)] $\beta_f(u^{_{p}}+\varepsilon\vec e,v^{_{1}}-\varepsilon\vec e)=\beta_f(\bar u,\bar v)$, for $\varepsilon>0$ sufficiently small. 
%Indeed, we can use again the right continuity of $\beta_f(u,v)$ in $u$ and $v$.
\end{description}

Thus, by applying (F1) and (F2) we get 
\begin{eqnarray*}
\sum_{\newatop{i=1,\dots,p}{j=1,\dots,q}}\mu_{f,\varepsilon\vec e\,}(u^{_{i}},v^{_{j}}\,)&=&\beta_{f}(\vu^{_p}+\varepsilon\vec e,\vv^{_1}-\varepsilon\vec e\,)-\beta_{f}(\vu^{_1}-\varepsilon\vec e,\vv^{_1}-\varepsilon\vec e\,)\\
&+&\beta_{f}(\vu^{_1}-\varepsilon\vec e,\vv^{_q}+\varepsilon\vec e\,)-\beta_{f}(\vu^{_p}+\varepsilon\vec e,\vv^{_q}+\varepsilon\vec e\,)=\\
&=&\beta_{f}(\vu^{_p}+\varepsilon\vec e,\vv^{_1}-\varepsilon\vec e\,)-\beta_{f}(\vu^{_p}+\varepsilon\vec e,\vv^{_q}+\varepsilon\vec e\,),
\end{eqnarray*}
$\varepsilon$ being a sufficiently small positive real number. Analogously, by applying (F1), (F2) and (F3) we get 
\begin{eqnarray*}
\sum_{i=1,\dots,p}\mu_{f,\varepsilon\vec e\,}^{_{\infty}}(u^{_{i}},v^{_{\varepsilon}})&=&\beta_{f}(\vu^{_p}+\varepsilon\vec e,v^{_{\varepsilon}})-\beta_{f}(\vu^{_1}-\varepsilon\vec e,v^{_{\varepsilon}})=\\
&=&\beta_{f}(\vu^{_p}+\varepsilon\vec e,\vv^{_{\varepsilon}})=\beta_{f}(\vu^{_p}+\varepsilon\vec e,\vv^{_q}+\varepsilon\vec e\,).
\end{eqnarray*}
Indeed, all  other terms cancel out, as shown in Figure~\ref{representationFigure}.} 
\begin{figure}
\psfrag{u}{$\bar u$}\psfrag{u1}{$u^{_1}$}\psfrag{u2}{$u^{_2}$}\psfrag{us}{$u^{_p}$}\psfrag{v}{$\bar v$}\psfrag{v1}{$v^{_1}$}\psfrag{v2}{$v^{_2}$}\psfrag{vt}{$v^{_q}$}\psfrag{u+e}{$\bar u+\vec e$}\psfrag{v+e}{$\bar v+\vec e$}
\includegraphics[width=0.75\textwidth]{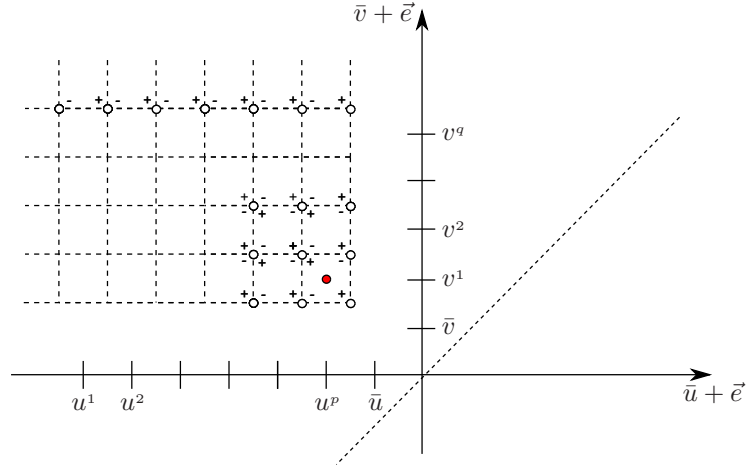}\caption{The idea behind the cancellation process used in the proof of Theorem~\ref{representation}.}\label{representationFigure}
\end{figure}
So we can write
\begin{equation*}
\sum_{\newatop{i=1,\dots,p}{j=1,\dots,q}}\mu_{f,\varepsilon\vec e\,}(u^{_{i}},v^{_{j}})
+\sum_{i=1,\dots,p}\mu_{f,\varepsilon\vec e\,}^{_{\infty}}(u^{_{i}},v^{_{\varepsilon}})=\beta_{f}(\vu^{_p}+\varepsilon\vec e,\vv^{_1}-\varepsilon\vec e\,).
\end{equation*} 

It is not restrictive to assume $\varepsilon$ sufficiently small such that, by the right continuity of $\beta_f(u,v)$ in $u$ and $v$, $\beta_f(u^{_{p}}+\varepsilon\vec e,v^{_{1}}-\varepsilon\vec e)=\beta_f(\bar u,\bar v)$ thus getting the claim.

\end{proof}

We end this section with two results, showing that discontinuity points of $\beta_f$ propagate from cornerpoints, both proper and at infinity. For what follows, it is convenient to observe that equality (\ref{mainFormula}) can be reformulated as 
\begin{equation}\label{mainFormula2}
\beta_f(\bar u,\bar v)=\sum_{s\geq 0,t>0}\mu_{f}(\bar u-s\vec e,\bar\vv+t\vec e\,)+\sum_{s\geq 0}\mu_{f}(\bar u-s\vec e,\infty).
\end{equation}

\begin{prop}\label{cornerp=>disc}
If $(\bar\vu,\bar \vv)\in D_n^+$ is a proper cornerpoint of $\beta_{f}$, then the following statements hold:
\begin{enumerate}
\item[$(i)$] $\bar\vu$ is a discontinuity point of $\beta_{f}(\cdot,\hat\vv)$, for every $\hat\vv$ with $\bar\vu\prec\hat\vv\prec\bar{\vv}$; %and $\hat\vv\in\langle \bar{\vv}-\bar{\vu}\rangle$.
\item[$(ii)$] $\bar\vv$ is a discontinuity point of $\beta_{f}(\hat\vu,\cdot)$, for every $\hat\vu$ with $\bar{\vu}\preceq\hat \vu\prec \bar{\vv}$. %and $\hat\vu\in\langle \bar{\vv}-\bar{\vu}\rangle$;
\end{enumerate}
\end{prop}

\begin{proof}
Let us  prove assertion $(i)$. Fix $\hat\vv$ such that $\bar\vu\prec\hat\vv\prec\bar{\vv}$. Let $\vec e\in\R^n$, with $\vec e\succ 0$ and $\|\vec e\,\|_{\infty}$ sufficiently small so that $\bar u+\vec e\prec\hat v\prec\bar v-\vec e$. By applying the Multidimensional Representation Theorem~\ref{representation}, and recalling (\ref{mainFormula2}), we get {\setlength\arraycolsep{2pt}
$$
\beta_{f}(\bar u+\vec e,\bar v-\vec e\,) - \beta_{f}(\bar u-\vec e,\bar v-\vec e\,)  = \underset{\newatop{-1\leq s< 1}{t > -1}}{\sum}\mu_{f}(\bar u-s\vec e,\bar v+t\vec e\,) + \underset{-1\leq s< 1}{\sum}\mu_{f}(\bar u-s\vec e,\infty).
$$
Now, $\mu_{f}(\bar\vu,\bar \vv)$ is an addend of the first sum in the above equality. Since $\mu_{f}(\bar\vu,\bar \vv)>0$, we have $\beta_{f}(\bar u+\vec e,\bar v-\vec e)-\beta_{f}(\bar u-\vec e,\bar v-\vec e)>0$. The Multidimensional Jump Monotonicity Lemma~\ref{Jump} implies that $\beta_{f}(\bar u+\vec e,\hat v)-\beta_{f}(\bar u-\vec e,\hat v)\geq\beta_{f}(\bar u+\vec e,\bar v-\vec e)-\beta_{f}(\bar u-\vec e,\bar v-\vec e)$. By the arbitrariness of $\vec e$, we deduce that $\bar u$ is a discontinuity point of $\beta_{f}(\cdot,\hat\vv)$. The proof of assertion $(ii)$ is analogous.}
\end{proof}

\begin{prop}\label{cornerp@infty=>disc}
If $(\bar\vu,\infty)\in D_n^*$ is a cornerpoint at infinity of $\beta_{f}$, then $\bar{\vu}$ is a discontinuity point of $\beta_{f}(\cdot,\hat \vv)$ for every $\bar{\vu}\prec\hat\vv$.
\end{prop}
\begin{proof}
Analogous to that of Proposition~\ref{cornerp=>disc}$(i)$.
%Take $\hat\vv\succ\bar u$. Let $\vec e\in\R^n$, with $\vec e\succ 0$ and $\|\vec e\,\|_{\infty}$ sufficiently small so that $\bar\vu+\vec e\prec\hat\vv$. By the Multidimensional Representation Theorem~\ref{representation}, and recalling (\ref{mainFormula2}), we have  
%$$
%\beta_{f}(\bar\vu+\vec e,\hat\vv)-\beta_{f}(\bar\vu-\vec e,\hat\vv)=\underset{\newatop{-1\leq s< 1}{t > 0}}{\sum}\mu_{f}(\bar u-s\vec e,\hat v+t\vec e\,) + \underset{-1\leq s< 1}{\sum}\mu_{f}(\bar u-s\vec e,\infty).
%$$
%Since $\mu_{f}(\bar\vu,\infty)>0$, $\beta_{f}(\bar\vu+\vec e,\hat\vv)-\beta_{f}(\bar\vu-\vec e,\hat\vv)>0$. By the arbitrariness of $\vec e$, we deduce that $\bar\vu$ is a discontinuity point of $\beta_{f}(\cdot,\hat\vv)$.
\end{proof}

\section{Stability of persistence spaces}\label{stability}
In this section we prove that small changes in the considered  functions induce not greater changes in the corresponding persistence
spaces. In particular, the distance between two functions $f,g:X\to\R^n$ is measured by $\max_{x\in X}\|f(x)-g(x)\|_{\infty}$, while the distance between $\spc(f)$ and $\spc(g)$ is measured according to the Hausdorff distance induced on $\overline{D_n^*}$ by the $\max$-norm:
$$
d_{H}(\spc(f),\spc(g))=\max\{\sup_{p\in\spc(f)}\inf_{q\in\spc(g)}\|p-q\|_{\infty},\sup_{q\in\spc(g)}\inf_{p\in\spc(f)}\|p-q\|_{\infty}\},
$$
where, for $p=(u,v),q=(u^{_{'}},v^{_{'}})\in\overline{D_n^*}$, we set 
\begin{equation}\label{infNorm}
\|p-q\|_{\infty}=\max\left\{\|u-u^{_{'}}\|_{\infty},\|v-v^{_{'}}\|_{\infty}\right\},
\end{equation}
with the convention that $\infty-\infty=0$ and $v-\infty=\infty-v=\infty$ for 
every $v\in\R^n$. In this way, in $d_H(\spc(f),\spc(g))$, also recalling Remark~\ref{multiplicityrRem0}, points at infinity are compared only with other points at infinity. Moreover, a direct computation yields the following formula for the distance of a point $p=(u,v)\in D_n^+$ to $D_n$:
$$
\inf_{q\in D_n}\|p-q\|_{\infty}=\min_{i=1,\dots,n}\frac{v_i-u_i}{2}.
$$ 
In this setting, our stability result can be stated as follows.
\begin{thm}[Stability Theorem]\label{stabilitySpaces}
Let $f,g:X\to\R^n$ be continuous functions.  Then 
$$
d_{H}(\spc(f),\spc(g))\leq\max_{x\in X}\|f(x)-g(x)\|_{\infty}.
$$  
\end{thm}
The proof of this result is based on the next propositions holding for any continuous function $f:X\to\R^n$. For every $p=(u,v)\in D_n^+$ and every $\vec e\in\R^n$ with $\vec e\succeq 0$, we set 
$$\mathcal{L}_{\vec e}(p)=\{(u-s\vec e,v+t\vec e\,)\in \R^n\times\R^n |\,s,t\in\R\,,-1\leq s<1,-1<t\leq 1\}.$$  
Note that, if $u+\vec e\prec v-\vec e$, then $\mathcal{L}_{\vec e}(p)\subseteq D_n^+$.

\begin{prop}\label{Stability:Prop1}
Let $\bar p=(\bar u,\bar v)\in D_n^+$ and $\vec e\in\R^n$, with $\vec e\succ 0$ and $\bar u+\vec e\prec\bar v-\vec e$. Then
{\setlength\arraycolsep{2pt}
\begin{equation}\label{eq1}
\begin{array}{ccccc} 
\beta_{f}(\bar u+\vec e,\bar v-\vec e\,) & - & 
\beta_{f}(\bar u-\vec e,\bar v-\vec e\,) & + &\\
& - & \beta_{f}(\bar u+\vec e,\bar v+\vec e\,) & + & 
\beta_{f}(\bar u-\vec e,\bar v+\vec e\,)
\end{array}
\end{equation} 
is equal to the cardinality of the set $\mathcal{L}_{\vec e}(\bar p)\cap\spc(f)$, where proper cornerpoints of $\beta_f$ are counted with their multiplicity.} 
\end{prop}
\begin{proof}
It is sufficient to apply the Multidimensional Representation Theorem, and recall (\ref{mainFormula2}).
\end{proof}
Note that, by the finiteness of $\beta_f$ and the Multidimensional Jump Monotonicity Lemma~\ref{Jump}, the sum in (\ref{eq1}) returns necessarily a non-negative, integer number. Thus, the cardinality of the set $\mathcal{L}_{\vec e}(\bar p)\cap\spc(f)$ is finite, and corresponds to the number of proper cornerpoints of $\beta_f$, counted with their multiplicity, in $\mathcal{L}_{\vec e}(\bar p)$. 

\begin{prop}\label{Stability:Prop2}
Let $\bar p=(\bar u,\bar v)\in D_n^+$. A real value $\bar\eta>0$ exists such that, for every $\eta\in\R$ with $0\leq\eta\leq\bar\eta$, the following statements hold:
\begin{enumerate}
\item[$(i)$] For $\vec\eta=(\eta,\eta,\dots,\eta)\in\R^n$, the set $$\overline{\mathcal{L}_{\vec\eta}(\bar p)}=\{(\bar u-s\vec\eta,\bar v+t\vec\eta\,)\in\R^n\times\R^n |\,s,t\in\R\,,-1\le s\le 1,-1\le t\leq 1\}$$ is contained in $D_n^+$;
\item[$(ii)$] For every continuous function $g:X\to\R^n$ with $\max_{x\in X}\|f(x)-g(x)\|_{\infty}\leq\eta$, the persistence space $\spc(g)$ has exactly $\mu_f(\bar p)$ proper cornerpoints in $\overline{\mathcal{L}_{\vec\eta}(\bar p)}$, counted with multiplicity.
\end{enumerate}

\end{prop}
\begin{proof}
By Proposition~\ref{Weps} we know that a sufficiently small real number $\varepsilon >0$ exists such that the set $W_{2\varepsilon}(\bar p)$ is entirely contained in $D_n^+$, and does not contain any discontinuity point of $\beta_f$. Clearly, for every $\eta$ with $0\le\eta\leq\varepsilon$, the set $\overline{\mathcal{L}_{\vec\eta}(\bar p)}$ is contained in $D_n^+$ thus proving $(i)$.

In order to prove $(ii)$, let $\bar\eta$ be any real number such that $0<\bar\eta<\frac{\varepsilon}{2}$, and let $\eta\in\R$ be such that $0\leq\eta\leq\bar\eta$. 

If $\eta=0$, then $g=f$, $\overline{\mathcal{L}_{\vec\eta}(\bar p)}=\{\bar p\}$ and hence the claim follows. Otherwise, if $\eta>0$, let us take a sufficiently small real number $\nu$ with $0<\nu<\eta$. We have $\eta+\nu<\varepsilon$ and $2\eta+\nu<2\varepsilon$. Moreover, setting $\vec\nu =(\nu,\nu,\dots,\nu)\in\R^n$, it holds that $\bar u+2\vec\eta+\vec\nu\prec\bar v-2\vec\eta-\vec\nu$. Now, if $g:X\to\R^n$ is a continuous function for which $\max_{x\in X}\|f(x)-g(x)\|_{\infty}\leq\eta$, by \cite[Lemma~2.5]{CeDi*12} we get 
$$
\beta_f(\bar u+\vec\nu,\bar v-\vec\nu\,)\leq\beta_g(\bar u+\vec\eta+\vec\nu,\bar v-\vec\eta-\vec\nu\,)\leq\beta_f(\bar u+2\vec\eta+\vec\nu,\bar v-2\vec\eta-\vec\nu\,).
$$
\begin{figure}
\psfrag{u}{$\bar u$}\psfrag{n}{$\bar u+\vec\nu$}\psfrag{n+m}{$\bar u+\vec\eta+\vec\nu$}\psfrag{n+2m}{$\bar u+2\vec\eta+\vec\nu$}\psfrag{e}{$2\varepsilon$}
\psfrag{v}{$\bar v$}\psfrag{nv}{$\bar v-\vec\nu$}\psfrag{m+n}{$\bar v-\vec\eta-\vec\nu$}\psfrag{m+2n}{$\bar v-2\vec\eta-\vec\nu$}
\includegraphics[width=0.5\textwidth]{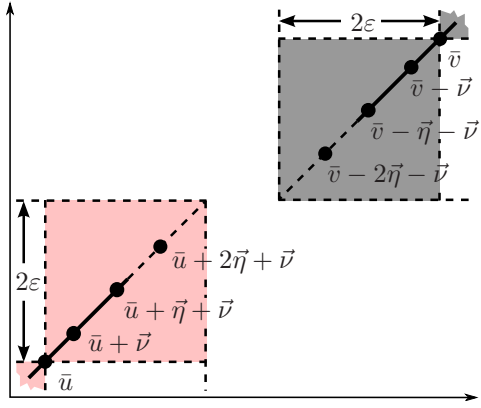}
\caption{Points relevant to the proof of Proposition~\ref{Stability:Prop2}. Bold segments represent part of the set $\mathcal{L}_{\vec\varepsilon}(\bar p)\subseteq W_{2\varepsilon}(\bar p)$ used in that proof.}\label{fig:Leta}
\end{figure}
Since $\beta_f$ is constant on each connected component of $W_{2\varepsilon}(\bar p)$, and $(\bar u+\vec\nu,\bar v-\vec\nu\,)$ and $(\bar u+2\vec\eta+\vec\nu,\bar v-2\vec\eta-\vec\nu\,)$ belong to the same connected component of $W_{2\varepsilon}(\bar p)$, as illustrated in Figure~\ref{fig:Leta}, we have 
$$
\beta_f(\bar u+\vec\nu,\bar v-\vec\nu\,)=\beta_f(\bar u+2\vec\eta+\vec\nu,\bar v-2\vec\eta-\vec \nu\,),
$$
thus implying that $\beta_f(\bar u+\vec\eta+\vec\nu,\bar v-\vec\eta-\vec\nu\,)=\beta_g(\bar u+\vec\eta+\vec\nu,\bar v-\vec\eta-\vec\nu\,)$. Analogously, $\beta_f(\bar u-\vec\eta-\vec\nu,\bar v-\vec\eta-\vec\nu\,)=\beta_g(\bar u-\vec\eta-\vec\nu,\bar v-\vec\eta-\vec\nu\,)$, $\beta_f(\bar u+\vec\eta+\vec\nu,\bar v+\vec\eta+\vec\nu\,)=\beta_g(\bar u+\vec\eta+\vec\nu,\bar v+\vec\eta+\vec\nu\,)$ and $\beta_f(\bar u-\vec\eta-\vec\nu,\bar v+\vec\eta+\vec\nu\,)=\beta_g(\bar u-\vec\eta-\vec\nu,\bar v+\vec\eta+\vec\nu\,)$. 

Therefore, by Proposition~\ref{Stability:Prop1} applied both to $g$ and $f$, the number of proper cornerpoints of $\beta_g$  in $\mathcal{L}_{\vec\eta+\vec\nu}(\bar p)$, counted with multilicity, is equal to that of $\beta_f$. On the other hand,  ${\mathcal{L}_{\vec\eta+\vec\nu}(\bar p)}\subseteq\mathcal{L}_{\vec\varepsilon}(\bar p)$. Hence, by Proposition~\ref{cornerp=>disc}, recalling that $W_{2\varepsilon}(\bar p)$ does not contain any discontinuity point of $\beta_f$, no proper cornerpoint of $\beta_f$ is in ${\mathcal{L}_{\vec\eta+\vec\nu}(\bar p)}$, except possibly for $\bar p$. In conclusion,  $\mu_f(\bar p)$ equals the number of proper cornerpoints of $\beta_g$, counted with multiplicity, contained in the set $\mathcal{L}_{\vec\eta+\vec\nu}(\bar p)$.

This is true for every sufficiently small $\nu>0$. Therefore, $\mu_f(\bar p)$ is equal to the number of cornerpoints of $\beta_g$ contained in the intersection $\bigcap_{\nu>0}\mathcal{L}_{\vec\eta+\vec\nu}(\bar p)$, thus proving the claim.
\end{proof}

Before going on, we remark that the proposed proof of Proposition~\ref{Stability:Prop2} cannot work without assuming that $\vec\varepsilon$, $\vec\eta$ and $\vec\nu$ are multiples of $\vec 1=(1,1,\dots,1)\in\R^n$. The obstruction is in the application of \cite[Lemma~2.5]{CeDi*12}. Also, such requirement ensures that the intersection $\bigcap_{\nu>0}\mathcal{L}_{\vec\eta+\vec\nu}(\bar p)$ in the end of the proof is over a sequence of nested sets, thus implying the claim. These remarks actually mirror the fact that Proposition~\ref{Stability:Prop2} does not hold without our assumptions on $\vec\eta$, as shown by the following example where we take $\vec\eta\in\R^2$ with $\vec\eta=\left(\eta,\frac{\eta}{2}\right)$. Similar examples for which Proposition~\ref{Stability:Prop2} does not hold can be built for every $\vec\eta\neq (\eta,\eta)$.

\begin{example}\label{Stability:Example}
Let $X$ be the closed interval $[0,1]$, and let $f,g:X\to\R^2$ be two functions linearly interpolating the following values: $f(0)=g(0)=(0,0)$, $f(1)=g(1)=(0,-1)$, $f\left(\frac{1}{2}\right)=(2,1)$ and $g\left(\frac{1}{2}\right)=(2,1+\eta)$, with $\eta>0$, as depicted in Figure~\ref{fig:Example}. We have $\max_{x\in X}\|f(x)-g(x)\|_{\infty}=\eta$. 

Set $\bar u=(0,0)$ and $\bar v=(2,1)$. According to Definition~\ref{multiplicityp}, the point $\bar p=(\bar u,\bar v)$ has multiplicity $\mu_f(\bar p)$ equal to 1 (we are considering 0$th$ homology). Hence, $\bar p\in\spc(f)$. Suppose now $\eta$ sufficiently small so that, taking $\vec\eta=\left(\eta,\frac{\eta}{2}\right)$, the set $\overline{\mathcal{L}_{\vec\eta}(\bar p)}$ is entirely contained in $D_n^+$. In this case it is easy to check that, in contrast with Proposition~\ref{Stability:Prop2}, $\overline{\mathcal{L}_{\vec\eta}(\bar p)}$ does not contain points of $\spc(g)$.      
\begin{figure}
\psfrag{u}{$\bar u$}\psfrag{1+e}{$1+\eta$}\psfrag{2}{2}
\psfrag{v}{$\bar v$}\psfrag{-1}{-1}\psfrag{0}{0}\psfrag{1}{1}
\includegraphics[width=\textwidth]{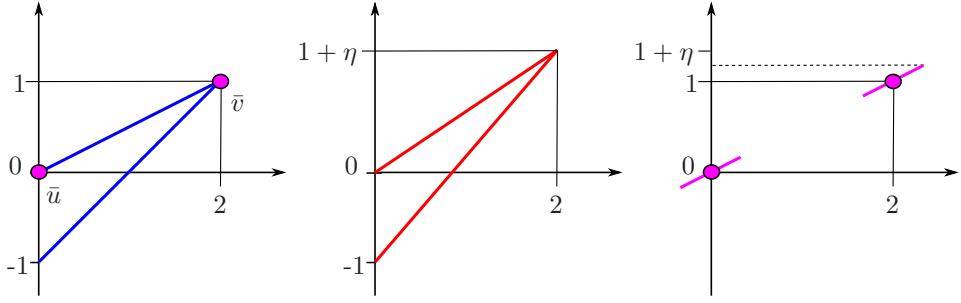}
\caption{$\I f$ (bold blue line, left), $\I g$ (bold red line, center) and a representation of the set $\overline{\mathcal{L}_{\vec\eta}(\bar p)}$ (bold purple segments, right) as defined in Example~\ref{Stability:Example}.}
\label{fig:Example}
\end{figure}
\end{example}
 
\begin{prop}\label{stabilityPoints}
Let $f,g:X\to\R^n$ be two continuous functions such that $\max_{x\in X}\|f(x)-g(x)\|_{\infty}\leq\varepsilon$. Then, for every proper cornerpoint $p\in\spc(f)$, a point $q\in\spc(g)$ exists such that $\|p-q\|_{\infty}\leq\varepsilon$.
\end{prop}
\begin{proof}
Let $p\in\spc(f)$. If a point $q\in D_n$ exists, for which $\|p-q\|_{\infty}\leq\varepsilon$, then there is nothing to prove because, by definition, $q\in\spc(g)$. Hence, let us assume that $\|p-q\|_{\infty}>\varepsilon$ for all $q\in D_n$. For every $\tau\in[0,\varepsilon]$, let $h_{\tau}$ be the function defined as $h_{\tau}=\frac{\varepsilon-\tau}{\varepsilon}f+\frac{\tau}{\varepsilon}g$. Note that, for every $\tau,\tau^{_{'}}\in[0,\varepsilon]$, we have $\max_{x\in X}\|h_{\tau}(x)-h_{\tau^{_{'}}}(x)\|_{\infty}\leq |\tau-\tau^{_{'}}|$. 

For $\tau\in[0,\varepsilon]$, let $\vec\tau=(\tau,\tau,\dots,\tau)\in\R^n$. Since $\|p-q\|_{\infty}>\varepsilon$ for all $q\in D_n$, we have that $\overline{\mathcal{L}_{\vec\tau}(p)}\subset D_n^+$ for every $\tau\in[0,\varepsilon]$. Now, consider the set 
$$
A=\left\{\tau\in[0,\varepsilon]:\exists\,q_{\tau}\in\spc(h_{\tau})\cap\overline{\mathcal{L}_{\vec\tau}(p)}\right\}.
$$
$A$ is non-empty, since $0\in A$. Let us set $\tau_*=\sup A$ and show that $\tau_*\in A$. Indeed, let $(\tau_j)$ be a sequence in $A$ converging to $\tau_*$. Since $\tau_j\in A$, for each $j$ there is a cornerpoint $q_j\in\spc(h_{\tau_j})$ such that $q_j\in\overline{\mathcal{L}_{\vec\tau_j}(p)}\subseteq\overline{\mathcal{L}_{\vec\varepsilon}(p)}$. By the compactness of $\overline{\mathcal{L}_{\vec\varepsilon}(p)}$, possibly by extracting a convergent subsequence, we can define $q=\lim_j q_j$. 

We have that $q\in\overline{\mathcal{L}_{\vec\tau_*}(p)}$. Indeed, if $q\not\in\overline{\mathcal{L}_{\vec\tau_*}(p)}$, then for every sufficiently large index $j$, we have $q_j\not\in\overline{\mathcal{L}_{\vec\tau_*}(p)}$. On the other hand, since $\tau_j\leq\tau_*$ for all $j$, it holds that $\overline{\mathcal{L}_{\vec\tau_j}(p)}\subseteq\overline{\mathcal{L}_{\vec\tau_*}(p)}$, thus giving a contradiction.

Moreover, the multiplicity $\mu_{h_{\tau_*}}(q)$ of $q$ for $\beta_{h_{\tau_*}}$ is strictly positive. Indeed, since $\tau_j\to\tau_*$ and $q\in\overline{\mathcal{L}_{\vec\tau_*}(p)}$, for every arbitrarily small $\eta>0$ and any sufficiently large $j$, the set $\overline{\mathcal{L}_{\vec\eta}(q)}$, with $\vec\eta=(\eta,\eta,\dots,\eta)\in\R^n$, contains at least one cornerpoint $q_j\in\spc(h_{\tau_j})$. But Proposition~\ref{Stability:Prop2} implies that, for each sufficiently small $\eta >0$, the set $\overline{\mathcal{L}_{\vec\eta}(q)}$ contains exactly as many cornerpoints of $\spc(h_{\tau_j})$ as $\mu_{h_{\tau_*}}(q)$, provided that $|\tau_j-\tau_*|\leq\eta$. Therefore, the multiplicity $\mu_{h_{\tau_*}}(q)$ is strictly positive, thus implying that $\tau_*\in A$.

To conclude the proof, we have to show that $\max A=\varepsilon$. If $\tau_*<\varepsilon$, by using Proposition~\ref{Stability:Prop2} once again we see that there exist a real value $\eta >0$ with $\tau_*+\eta<\varepsilon$, and a point $q_{\tau_*+\eta}\in\spc(h_{\tau_*+\eta})$ for which $q_{\tau_*+\eta}\in\overline{\mathcal{L}_{\vec\eta}(q)}$. Consequently, $\|q-q_{\tau_*+\eta}\|_{\infty}\leq\eta$. Thus, by the triangular inequality we would have $\|p-q_{\tau_*+\eta}\|_{\infty}\leq\tau_*+\eta$ and hence $q_{\tau_*+\eta}\in\overline{\mathcal{L}_{\vec\tau_*+\vec\eta}(p)}$, implying that $\tau_*+\eta\in A$. Obviously, this would contradict the fact that $\tau_*=\max A$. Therefore, $\varepsilon=\max A$, so that $\varepsilon\in A$. Clearly, this proves the claim.
\end{proof}

In analogy to proper points, we prove the equivalent of Propositions~\ref{Stability:Prop1}, ~\ref{Stability:Prop2} and \ref{stabilityPoints} for points at infinity. For every $p=(u,\infty)\in D_n^*$ and every $\vec e\in\R^n$ with $\vec e\succeq 0$, we set 
$$
\mathcal{N}_{\vec e}(p)=\{(u-s\vec e,\infty)\in D_n^*|-1\leq s< 1\}. 
$$
Further, for every real value $\varepsilon>0$ we denote by $v^{_{\varepsilon}}$ the $n$-tuple $(\varepsilon^{_{-1}},\varepsilon^{_{-1}},\dots,\varepsilon^{_{-1}})$. 

\begin{prop}\label{Stability:Prop4}
Let $\bar p=(\bar u,\infty)\in D_n^*$. A sufficiently small real value $\varepsilon>0$ exists such that $\bar u+\vec e\prec v^{_{\varepsilon}}$ for all $\vec e\in\R^n$ with $\vec e\succ 0$ and $\|\vec e\,\|_{\infty}<\varepsilon$. Moreover, for every $v\succ v^{_{\varepsilon}}$,
\begin{equation}\label{eq10}
\beta_f(\bar u+\vec e,v)-\beta_f(\bar u-\vec e,v)
\end{equation}
is equal to the cardinality of the set $\mathcal{N}_{\vec e}(\bar p)\cap\spc(f)$, where cornerpoints at infinity of $\beta_f$ are counted with their multiplicity.
\end{prop}
\begin{proof}
Clearly, a sufficiently small real value $\varepsilon>0$ exists for which $\bar u+\vec e\prec v^{_{\varepsilon}}$ whenever $\vec e\succ 0$ and $\|\vec e\,\|_{\infty}<\varepsilon$. For every $v\succ v^{_{\varepsilon}}$, by applying the Multidimensional Representation Theorem and using (\ref{mainFormula2}) we obtain
\begin{equation}\label{cpInfStab}
\beta_f(\bar u+\vec e,v)-\beta_f(\bar u-\vec e,v)=\sum_{\newatop{-1\leq s< 1}{t>0}}\mu_f(\bar u-s\vec e,v+t\vec e\,)+\sum_{-1\leq s< 1}\mu_f(\bar u-s\vec e,\infty).
\end{equation}
Now, suppose $\varepsilon$ is small enough so that $\varepsilon^{-1}\geq\max_{x\in X}\|f(x)\|_{\infty}$. It follows that $X\langle f\preceq v^{_{\varepsilon}}\rangle=X$. Hence, if $v\succ v^{_{\varepsilon}}$, then  $v$ cannot be a discontinuity point of $\beta_f(u,\cdot)$, for any $u\prec v$. By Proposition~\ref{cornerp=>disc}, this implies that the first sum in (\ref{cpInfStab}) runs over no proper cornerpoints, and hence all its terms vanish. This proves the claim.
\end{proof}
The finiteness and the monotonicity of $\beta_f$ imply that the sum in (\ref{eq10}) results in a non-negative integer number. Hence, the cardinality of the set $\mathcal{N}_{\vec e}(\bar p)\cap\spc(f)$ is finite, and corresponds to the number of cornerpoints at infinity of $\beta_f$, counted with their multiplicity, in $\mathcal{N}_{\vec e}(\bar p)$.
   
\begin{prop}\label{Stability:Prop5}
Let $\bar p=(\bar u,\infty)\in D_n^*$. A real value $\bar\eta>0$ exists such that, for every $\eta\in\R$ with $0\leq\eta\leq\bar\eta$ and every continuous function $g:X\to\R^n$ with $\max_{x\in X}\|f(x)-g(x)\|_{\infty}\leq\eta$, the persistence space $\spc(g)$ has exactly $\mu_f(\bar p)$ cornerpoints at infinity, counted with multiplicity, in the set
$$\overline{\mathcal{N}_{\vec\eta}(\bar p)}=\{(\bar u-s\vec\eta,\infty)\in D_n^*|-1\leq s\leq 1\},$$
with $\vec\eta=(\eta,\eta\dots,\eta)\in\R^n$.
\end{prop}
\begin{proof}
By Proposition~\ref{Veps}, a sufficiently small $\varepsilon >0$ exists such that the set $V_{2\varepsilon}(\bar p)$ is entirely contained in $D_n^+$, it does not contain any discontinuity point of $\beta_f$ and, setting $\vec\varepsilon=(\varepsilon,\varepsilon,\dots,\varepsilon)$, $\bar u+2\vec\varepsilon\prec v^{_{\varepsilon}}$. Proposition~\ref{cornerp@infty=>disc} implies that $\beta_f$ has no cornerpoints at infinity in $\mathcal{N}_{\vec\varepsilon}(\bar p)$, except possibly for $\bar p$.

Take $\bar v\in\R^n$ such that $\bar v-\frac{\vec\varepsilon}{2}\succ v^{_{\varepsilon}}$. Let $\bar\eta$ be any real number for which $0<\bar\eta<\frac{\varepsilon}{2}$, and let $\eta\in\R$ be such that $0\leq\eta\leq\bar\eta$. 

If $\eta=0$ then $g=f$, $\overline{\mathcal{N}_{\vec\eta}(\bar p)}=\{\bar p\}$ and hence the claim follows. 

Otherwise, if $\eta>0$, let us consider a sufficiently small $\nu\in\R$ with $0<\nu<\eta$. We have $\eta+\nu<\varepsilon$ and $2\eta+\nu<2\varepsilon$. Moreover, setting $\vec\nu =(\nu,\nu,\dots,\nu)\in\R^n$, it holds that $\bar u+2\vec\eta+\vec\nu\prec\bar v-\vec\eta$. Now, if $g:X\to\R^n$ is a continuous function for which $\max_{x\in X}\|f(x)-g(x)\|_{\infty}\leq\eta$, by \cite[Lemma~2.5]{CeDi*12} we get 
$$%\begin{equation}
\beta_f(\bar u+\vec\nu,\bar v+\vec\eta\,)\leq\beta_g(\bar u+\vec\eta+\vec\nu,\bar v)\leq\beta_f(\bar u+2\vec\eta+\vec\nu,\bar v-\vec\eta\,).
$$%\end{equation}
Note that $(\bar u+\vec\nu,\bar v+\vec\eta\,)$ and $(\bar u+2\vec\eta+\vec\nu,\bar v-\vec\eta\,)$ belong to the same connected component of $V_{2\varepsilon}(\bar p)$. Since $\beta_f$ is constant on each connected component of $V_{2\varepsilon}(\bar p)$, we have 
$$%\begin{equation}
\beta_f(\bar u+\vec\nu,\bar v+\vec\eta\,)=\beta_f(\bar u+2\vec\eta+\vec\nu,\bar v-\vec\eta\,),
$$%\end{equation}
thus implying that $\beta_f(\bar u+\vec\eta+\vec\nu,\bar v)=\beta_g(\bar u+\vec\eta+\vec\nu,\bar v)$. Analogously, $\beta_f(\bar u-\vec\eta-\vec\nu,\bar v)=\beta_g(\bar u-\vec\eta-\vec\nu,\bar v)$. 

Since $\beta_f$ has no cornerpoints at infinity in $\mathcal{N}_{\vec\varepsilon}(\bar p)$ except possibly for $\bar p$, by Proposition~\ref{Stability:Prop4} and the previous equalities we get that $\mu_f(\bar p)$ equals the number of cornerpoints at infinity of $\beta_g$ contained in the set $\mathcal{N}_{\vec\eta+\vec\nu}(\bar p)$. This is true for every sufficiently small $\nu>0$. Therefore, $\mu_f(\bar p)$ is equal to the number of cornerpoints at infinity of $\beta_g$ contained in the intersection $\underset{\nu>0}{\cap}\mathcal{N}_{\vec\eta+\vec\nu}(\bar p)$, thus proving the claim.
\end{proof}

\begin{prop}\label{stabilityLines}
Let $f,g:X\to\R^n$ be two continuous functions such that $\max_{x\in X}\|f(x)-g(x)\|_{\infty}\leq\varepsilon$. Then, for every cornerpoint at infinity $p=(u,\infty)\in\spc(f)$, a point $q=(u^{_{'}},\infty)\in\spc(g)$ exists such that $\|p-q\|_{\infty}\leq\varepsilon$.
\end{prop}
\begin{proof}
The proof is analogous to that of Proposition~\ref{stabilityPoints}, after noting that the norm $\|\cdot\|_{\infty}$ introduced in (\ref{infNorm}) naturally induces a topology on $D_n^*$.  
\end{proof}

We are now ready to prove the stability of persistence spaces.

\begin{proof}[Proof of Therorem~\ref{stabilitySpaces}]
Let $\max_{x\in X}\|f(x)-g(x)\|_{\infty}=\varepsilon$. Proposition~\ref{stabilityPoints} and Proposition~\ref{stabilityLines} imply that $\sup_{p\in\spc(f)}\inf_{q\in\spc(g)}\|p-q\|_{\infty}\leq\varepsilon$. Moreover, by exchanging the roles of $f$ and $g$, once more by Propositions~\ref{stabilityPoints} and ~\ref{stabilityLines} we also get $\sup_{q\in\spc(g)}\inf_{p\in\spc(f)}\|p-q\|_{\infty}\leq\varepsilon$. Thus $d_{H}(\spc(f),\spc(g))\leq\varepsilon$, and the claim follows.
\end{proof}

\section{The points of the persistence space are pairs of homological critical values} \label{hcv}

In this section we recall the concept of homological critical value  for a vector-valued continuous function from \cite{CaEt*12}, and we show that the points of a persistence space are pairs of homological critical values.

\begin{defi}\label{homcritvalue}
We shall say that $\vu\in\R^n$ is a \emph{homological critical value} for $f:X\to\R^n$  if there exists an integer number $k$ such that, for all
sufficiently small real values  $\eps > 0$, two values $\vu',\vu''\in\R^n$ can be found with $\vu'\preceq\vu\preceq\vu''$, $\|\vu'-\vu\|_\infty<\eps$,  $\|\vu''-\vu\|_\infty<\eps$, such that the homomorphism ${\iota}_k^{\vu', \vu''}: \check{H}_k(X\langle\fr \preceq\vu'\rangle)\rightarrow\check{H}_k(X\langle \fr \preceq\vu''\rangle)$ induced by inclusion is not an isomorphism.
\end{defi}

Let us observe that homological critical values of a vector-valued function $\fr$ do not necessarily belong to the image of $\fr$.

\begin{prop}\label{disc=>homcrit}
Let $(u,v)\in D_n^+$. The following statements hold:
\begin{enumerate}
\item[$(i)$] If $\vu$ is a discontinuity of
$\beta_{f}(\cdot,\vv)$, then $\vu$ is a homological critical value of $f$;
\item[$(ii)$] If $\vv$ is a discontinuity of
$\beta_{f}(\vu,\cdot)$, then $\vv$ is a homological critical value  of $f$.
\end{enumerate}
\end{prop}

\begin{proof}
Let us begin by proving $(i)$. If $\vu$ is a discontinuity point of $\beta_{f}(\cdot,\vv)$ in the homology degree $k$, then,
setting $Q_\eps(\vu)=\{\vu'\in\R^n: \|\vu'-\vu\|_\infty<\eps\}$, for every real value $\eps>0$ so small that $Q_\eps(\vu)\times \{\vv\}\subseteq D_n^+$,
there exists $\vu'\in Q_\eps(\vu)$ such that
$\beta_{f}(\vu',\vv)\neq\beta_{f}(\vu,\vv)$.  

In the case when $\beta_{f}(\vu',\vv)<\beta_{f}(\vu,\vv)$, it is not restrictive to assume that $\vu'\prec\vu$. Indeed, if this were not so, we could still find some $\vu''\in Q_\eps(\vu)$ with $\vu''\prec\vu'$ and $\vu''\prec\vu$. By Proposition~\ref{Monotonicity}, we would have $\beta_{f}(\vu'',\vv)<\beta_{f}(\vu,\vv)$ and could take $u''$ in place of $u'$.

Let us prove that the homomorphism ${\iota}_k^{u',u}:\check{H}_k(X\langle\fr\preceq u'\rangle)\rightarrow\check{H}_k(X\langle\fr\preceq u\rangle)$ induced by inclusion is not an isomorphism. By contradiction, let us suppose that ${\iota}_k^{u',u}$ is an isomorphism. Then, by the commutativity of the diagram
\begin{eqnarray*}
\begin{array}{c}\label{discontu}
\begin{centering}
\hfill\xymatrix {\check{H}_k(X\langle \fr \preceq u'\rangle)\ar[rr]^(.50){{\iota}_k^{u', u}}\ar[dr]_(.45){\!{\iota}_k^{u',
\vv}}&&\check{H}_k(X\langle \fr \preceq u\rangle)\ar[dl]^(.45){{\iota}_k^{u, \vv}}
\\
&\check{H}_k(X\langle \fr \preceq\vv\rangle)& }\hfill
\end{centering}
\end{array}
\end{eqnarray*}
it would follow that $\I{\iota}_k^{u', \vv}$ and
$\I{\iota}_k^{u, \vv}$ are isomorphic, thus contradicting
the assumption that $\beta_{f}(u',\vv)=\rk\,
{\iota}_k^{u', \vv}\neq \rk\, {\iota}_k^{u,
\vv}=\beta_{f}(u,\vv)$.
The case $\beta_{f}(\vu',\vv)>\beta_{f}(\vu,\vv)$ can be analogously handled.

A similar proof works for $(ii)$. If $\vv$ is a discontinuity point of $\beta_{f}(\vu,\cdot)$ in the homology degree $k$, then, considering the neighborhood $Q_\eps(\vv)$ of $\vv$, for every  real value $\eps>0$ so small that $\{\vu\}\times Q_\eps(\vv)\subseteq D_n^+$, there exists $\vv'\in Q_\eps(\vv)$ such that $\beta_{f}(\vu,\vv')\neq\beta_{f}(\vu,\vv)$.  

In the case when $\beta_{f}(\vu,\vv)<\beta_{f}(\vu,\vv')$, we can assume $\vv\prec\vv'$. By considering the commutative
diagram

\begin{eqnarray*}
\begin{array}{c}\label{discontv}
\begin{centering}
\hfill \xymatrix {&\check{H}_k(X\langle\fr
\preceq\vu\rangle)\ar[dl]_(.45){{\iota}_k^{\vu,v}}\ar[dr]^(.45){\!{\iota}_k^{\vu,v'}}&\\
\check{H}_k(X\langle \fr \preceq v\rangle)\ar[rr]^(.50){{\iota}_k^{v,v'}} &&\check{H}_k(X\langle\fr\preceq v'\rangle)}\hfill
\end{centering}
\end{array}
\end{eqnarray*}
we can prove, again by contradiction, that ${\iota}_k^{\vv,\vv'}$ is not an isomorphism. The case when $\beta_{f}(\vu,\vv)>\beta_{f}(\vu,\vv')$ is similar.
\end{proof}

\begin{thm}\label{cpt=>hcv}
If $(\bar{\vu},\bar{\vv})\in D_n^+$ is a proper cornerpoint of $\beta_{f}$, then both $\bar{\vu}$ and $\bar{\vv}$ are homological critical values for $f$. Moreover, if $(\bar{\vu},\infty)$, with $\bar{\vu}\in\R^n$ is a cornerpoint at infinity of $\beta_{f}$, then  $\bar{\vu}$ is a homological critical value for $f$.
\end{thm}
\begin{proof}
The first claim follows immediately from Proposition~\ref{cornerp=>disc} and Proposition~\ref{disc=>homcrit}, applied in this order. Analogously, the second claim follows  from
Proposition~\ref{cornerp@infty=>disc} and Proposition~\ref{disc=>homcrit}, applied in this order.
\end{proof}

We end this section with a further result about homological critical values, for which it is crucial to use a homology theory of \v{C}ech type. 

\begin{prop}
Let $X$ be a compact  space having a triangulation of dimension $d$, and let  $f:X\to \R^n$ be a continuous function. Then $f$ has no homological critical values for the homology degrees $k>d$.  In particular, $\beta_f$ is identically zero for $k>d$.    
\end{prop}

\begin{proof}
It is well known (cf. \cite[p. 320]{EiSt}) that, for   any compact pair $(Y,A)$ in $X$, the \v{C}ech homology theory with coefficients in a field ensures that $H_k(Y,A)=0$ for $k>d$. Applying this result with $Y=X\langle f\preceq u\rangle$ and $A=\emptyset$,  we obtain  that $H_k(X\langle f\preceq u\rangle)$ is trivial for every $u\in\R^n$ and every $k>d$.
\end{proof}

\section{Discussion}\label{conclusions}
We have presented a stable and complete representation of multidimensional persistent Betti numbers via persistence spaces. 

In our present treatment of persistence spaces we have focused on data belonging to the topological category. We briefly discuss here a different setting which have been treated in \cite{CeFr09} for the case of 0th homology, namely the case when the considered space $X$ and the function $f:X\to\R^n$ belong to the smooth category. The ideas contained in that work can be generalized to any homology degree in order to establish a link between persistence spaces and the concept of \emph{Pareto criticality}. 

More in detail, assume $X$ is a smooth, closed (i.e. compact without boundary), connected Riemannian manifold, and $f:X\to\R^n$ is a smooth function. A point $x\in X$ is a \emph{Pareto critical point} of $f$ if the convex hull of $\nabla f_{1}(x),\dots,\nabla f_{n}(x)$ contains the null vector. Moreover, $u\in\R^n$ is a \emph{Pareto critical value} of $f$ if $u=f(x)$ for some Pareto critical point $x\in X$. 

In this setting, following \cite{CeFr09} it is possible to  show that the points of $\spc(f)$ have coordinates that are projections of Pareto critical values of $f$.  

From the application viewpoint, an interesting case is when data come in the form of triangulated compact spaces endowed with interpolated functions. 

The discrete case of multidimensional persistent Betti numbers has been  treated in \cite{CaEt*12} so we refer the interested reader to that paper for further details. Here we confine ourselves to report that in that case the set of homological critical values of $f$ is a nowhere dense set in $\R^n$. Moreover its $n$-dimensional Lebesgue measure is zero. Finally it is worth mentioning that, although the set of homological critical values  may be an uncountable set even in the discrete setting, it admits  a finite representative set  as stated in \cite[Prop.~4.6]{CaEt*12}. 

These results suggest the following open question: in the discrete case, is it possible to determine  a finite representative for the corresponding persistence spaces? Clearly, a positive answer to this question would open the way to the definition of a bottleneck distance between these representative points.

\bibliographystyle{amsplain}
\bibliography{biblioABC}

\end{document}